\documentclass{amsart}
\usepackage{amsmath,amssymb,amsfonts}
\usepackage[mathscr]{eucal}
\usepackage{xcolor}
\usepackage[all,color]{xy}

\usepackage{enumerate}

\theoremstyle{plain}
\newtheorem{theorem}[equation]{Theorem}

\newtheorem{prop}[equation]{Proposition}

\theoremstyle{definition}
\newtheorem{definition}[equation]{Definition}
\newtheorem{example}[equation]{Example}
\newtheorem{question}[equation]{Question}

\newtheorem{remark}[equation]{Remark}

\numberwithin{equation}{section}

\input xy
\xyoption{all}

\newcommand{\Cat}{\mathcal Cat}
\newcommand{\css}{\mathcal{CSS}}
\newcommand{\Deltaop}{\Delta^{\op}}
\newcommand{\disc}{\operatorname{disc}}
\newcommand{\ev}{\operatorname{ev}}
\newcommand{\heq}{\operatorname{heq}}
\newcommand{\Ho}{\operatorname{Ho}}
\newcommand{\Hom}{\operatorname{Hom}}
\newcommand{\id}{\operatorname{id}}
\newcommand{\Map}{\operatorname{Map}}
\newcommand{\map}{\operatorname{map}}
\newcommand{\nerve}{\operatorname{nerve}}
\newcommand{\ob}{\operatorname{ob}}
\newcommand{\op}{\operatorname{op}}
\newcommand{\Qcat}{\qcat}
\newcommand{\qcat}{\mathcal{QC}at}

\newcommand{\Se}{\mathcal Se}
\newcommand{\Secat}{\mathcal Se\mathcal Cat}
\newcommand{\sesp}{\mathcal Se\mathcal Sp}
\newcommand{\Sets}{\mathcal Sets}
\newcommand{\Sing}{\operatorname{Sing}}
\newcommand{\snerve}{\operatorname{snerve}}
\newcommand{\SSets}{\mathcal{SS}ets}
\newcommand{\Thetanop}{\Theta_n^{\op}}

\newcommand{\Top}{\mathcal Top}
\newcommand{\umap}{\underline{\map}}
\newcommand{\uMap}{\underline{\Map}}

\begin{document}

\title[$(\infty,n)$-categories]{A survey of models for $(\infty, n)$-categories}

\author[J.E. Bergner]{Julia E. Bergner}

\address{Department of Mathematics, University of Virginia, Charlottesville, VA 22904}

\email{jeb2md@virginia.edu}

\date{\today}



\thanks{The author was partially supported by NSF grant DMS-1659931, the Isaac Newton Institute for Mathematical Sciences, Cambridge, during the program ``Homotopy Harnessing Higher Structures", supported by EPSRC grant no EP/K032208/1, and the AWM Michler Prize.}

\begin{abstract}
We give describe several models for $(\infty,n)$-categories, with an emphasis on models given by diagrams of sets and simplicial sets.  We look most closely at the cases when $n \leq 2$, then summarize methods of generalizing for all $n$.
\end{abstract}

\maketitle

\section{Introduction: What should an $(\infty,n)$-category be?}

When we say that we want to find ``models" for $(\infty,n)$-categories, we are looking for concrete mathematical objects which encode this desired structure.  But first, we need to answer a more basic question: What is an $(\infty,n)$-category anyway?  A short answer is that it is should be a higher category which is given up to homotopy, in some sense.  Alternatively, it is a higher category in which sufficiently high-level morphisms are invertible.  

To give a better description of what $(\infty, n)$-categories should be, let us first consider what is meant by a higher category more generally.  Recall that a category in the usual sense consists of objects, morphisms between objects, and a composition law for morphisms, such that each object has an identity morphism, and such that composition is associative.  As we move to higher categories, for specificity let us refer to morphisms as \emph{1-morphisms}.  

The essential idea behind a 2-category is that 1-morphisms which share source and target objects can have \emph{2-morphisms} between them: 
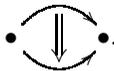
\[\xymatrix{{\bullet} \ar@/^1pc/[r]^{}="30" \ar@/_1pc/[r]^{}="31" \ar@{=>}"30";"31"
& {\bullet.}} \]
These 2-morphisms should themselves have a composition law which is associative; each 1-morphism should also have an associated identity 2-morphism.  One could then imagine extending these ideas for successively higher morphisms.  If we stop at some dimension $n$, we call the resulting structure an $n$-\emph{category}; if we continue to arbitrary $n$ we obtain an $\infty$-\emph{category}.

When the unitality of identity morphisms and associativity hold strictly, it is not hard to define these structures concretely.  A \emph{2-category} can be defined to be a category \emph{enriched in categories}.  In other words, it consists of objects, together with, for any pair of objects $(x,y)$, a category of morphisms from $x$ to $y$. The objects of this category define 1-morphisms $x \rightarrow y$, and the morphisms between them are 2-morphisms.  

More generally, we can define an $n$-\emph{category} to be a category enriched in $(n-1)$-categories.  Defining $\infty$-categories is more delicate, however.

The problem is that these strict conditions do not hold for many interesting examples.  Often associativity, for example, does not hold strictly, but rather only up to isomorphism.  For such a structure, we need to ask that various coherence laws hold, leading to the notion of a \emph{weak $n$-category}.   In the case of a weak 2-category, or \emph{bicategory}, the required coherence laws are manageable to write down.  The situation gets successively more complicated for higher weak $n$-categories.  While there are many different proposed definitions of such structures, they are, in general, not known to be equivalent to one another.  A good introduction to many of these approaches can be found in \cite{leinstersurvey}.

Much more progress has been made when we look at higher categories from a homotopical perspective, leading to the notion of an $(\infty,n)$-category.  Heuristically, an $(\infty,n)$-category should be a weak $\infty$-category such that all $i$-morphisms are weakly invertible for $i>n$.  Since we already suggested that strict $\infty$-categories are not so straightforward, never mind weak ones, it is not immediately clear that such a notion should be tractable in any way.  Yet, starting with $(\infty,0)$-categories, or $\infty$-groupoids, which can be defined much more easily, we can again use a process of enrichment to make an inductive definition: an $(\infty,n)$-category is a category enriched in $(\infty, n-1)$-categories.  As long as we have a good model for $(\infty,0)$-categories to begin, we have a very precise way to define higher $(\infty,n)$-categories.

However, just as we often want weak, rather than strict, $n$-categories, we really want $(\infty,n)$-categories to be defined via some kind of weak enrichment, so that composition might only be defined up to homotopy and the associativity and unit conditions need not hold in a strict way but rather in some more homotopical sense.  While it might seem that we are back to the same kinds of issues as before, these structures are remarkably more manageable than weak $n$-categories.  Such weaker models for $(\infty,1)$-categories can be defined using homotopy-theoretic tools such as simplicial objects.  The generalization of these structures to higher $(\infty,n)$-categories is difficult work and can be done in myriad ways, some of which are still work in progress, but we have a number of models which are well-understood and known to be equivalent to one another.

In this this paper, look at some of these models from a homotopy-theoretic point of view.  We look at the associated model category structures for each, and the comparisons between them in the form of Quillen equivalences.

After introducing $(\infty,0)$-categories (and in doing so, reviewing some basic ideas in simplicial homotopy theory) and various models for $(\infty,1)$-categories, we focus primarily on $(\infty,2)$-categories.  By emphasizing lower dimensions, we hope the reader can grasp the basic ideas without immediately drowning in the abstraction of the general case.  At the end of the paper, we discuss fully general $(\infty,n)$-categories, at least for some of the known definitions.

It is important to remark here that we by no means cover all models or approaches to the subject.  To keep this paper to a reasonable length, we have chosen to focus on models given by diagrams of sets or simplicial sets, and from the point of view of model category structures.  In particular, we do not look at the axiomatic treatments of To\"en \cite{toen} and of Barwick and Schommer-Pries \cite{bsp}, nor at the new model-independent approaches of Riehl and Verity \cite{rvbook}, \cite{rv}.  Other models we omit include the $n$-relative category model of Barwick and Kan \cite{bk}, \cite{bknrel}, the $n$-complicial sets of Verity \cite{verity} and their recent comparison with other kinds of diagrams \cite{or}, the more geometric approach of Ayala, Francis, and Rozenblyum \cite{afr}, and the variants building on marked simplicial sets of Lurie in \cite{luriegc}.  These papers comprise important work in the subject and we hope the introduction we have given here will inspire the reader to look into them more closely.

\section{Some model category background}

Our approach to $(\infty,n)$-categories is in the framework of model categories.  The definitions we make here do not need this extra structure, but from the point of view of homotopy theory it is valuable to have it nonetheless.  The reader only interested in the ways to define $(\infty,n)$-categories themselves can safely ignore this language and regard it as means for making the formal comparisons.  In this section, we summarize some of the model category language that is used.  

Let us start with a terse version of the definition; good references include the surveys \cite{ds} and \cite{gs} or the books \cite{hirsch} and \cite{hovey}.

\begin{definition}
A \emph{model category} is a category $\mathcal M$, possessing all small limits and colimits, together with three distinguished classes of morphisms, called \emph{weak equivalences}, \emph{fibrations}, and \emph{cofibrations}, satisfying several axioms.  An \emph{acyclic (co)fibration} is a (co)fibration which is also a weak equivalence.
\end{definition}

The existence of limits and colimits in $\mathcal M$ guarantees the existence of an initial object $\varnothing$ and a terminal object $\ast$.  

\begin{definition}
	An object $X$ of a model category $\mathcal M$ is \emph{cofibrant} if the unique morphism $\varnothing \rightarrow X$ is a cofibration.  Dually, $X$ is \emph{fibrant} if the unique morphism $X \rightarrow \ast$ is a fibration.
\end{definition}

The structure of a model category enables one to have a well-defined homotopy category without running into set-theoretic obstacles.  On the one hand, the critical information of a model category is the collection of weak equivalences; any two model categories with the same weak equivalences have equivalent homotopy categories.  On the other hand, the fibrant and cofibrant objects are important in the construction of the specific homotopy category associated to a given model category.

The following definition is our means of comparison between model categories.

\begin{definition}
	Let $\mathcal M$ and $\mathcal N$ be model categories.  A \emph{Quillen pair} is an adjoint pair of functors
	\[ F \colon \mathcal M \rightleftarrows \mathcal N \colon G \]
	such that the left adjoint $F$ preserves cofibrations and the right adjoint $G$ preserves fibrations.  It is a \emph{Quillen equivalence} if, additionally, a morphism $FX \rightarrow Y$ is a weak equivalence in $\mathcal N$ if and only if the corresponding morphism $X \rightarrow GY$ is a weak equivalence in $\mathcal M$.
\end{definition}

Throughout this paper, our goal is to describe model structures on certain categories in such a way that the fibrant and cofibrant objects model $(\infty, n)$-categories.  We then give Quillen equivalences between these
 model categories, which demonstrate that their corresponding models for $(\infty,n)$-categories really do capture the same structure.

Many of these models use the framework of simplicial sets and more general simplicial objects, so let us briefly review these ideas.  Let $\Delta$ be the category of finite ordered sets $[n] = \{0 \leq 1 \leq \cdots \leq n\}$ and order-preserving functions. 

\begin{definition}
A \emph{simplicial set} is a functor $\Deltaop \rightarrow \Sets$.
\end{definition}

Three critical examples of simplicial sets are the $n$-simplex $\Delta[n]$, which is the representable functor $\Hom_\Delta(-, [n])$, its boundary $\partial \Delta[n]$, in which the non-degenerate simplex in degree $n$ has been removed, and, for any $0 \leq k \leq n$, the $k$-horn $V[n,k]$, for which the simplex in degree $n-1$ opposite the vertex $k$ has also been removed.  

Given any simplicial set $K$, one can obtain from it a topological space $|K|$ via geometric realization.  This functor has a right adjoint, taking a topological space $X$ to the singular set $\Sing(X)$.  

\begin{theorem} \cite[II.3]{quillen} \label{ssetstop}
There is a model structure $\SSets$ on the category of simplicial sets with weak equivalences the maps whose geometric realizations are weak homotopy equivalences, and there is a model structure on the category of topological spaces in which the weak equivalences are the weak homotopy equivalences.  The adjoint pair $|-| \colon \SSets \leftrightarrows \Top \colon \Sing$ is a Quillen equivalence of model categories.
\end{theorem}

We also want to consider further structure on some of our model categories; we do not give full details of the definitions here but only the main idea. 

\begin{definition} \cite[9.1.6]{hirsch}
A model category $\mathcal M$ is \emph{simplicial} if, for any two objects $X$ and $Y$ of $\mathcal M$, there is a simplicial set $\Map(X,Y)$ of morphisms from $X$ to $Y$, and such that this simplicial structure is compatible with the model structure.
\end{definition}

\begin{definition} \cite[2.2]{rezktheta}
	A model category $\mathcal M$ is \emph{cartesian} if its underlying category is closed symmetric monoidal via the cartesian product and this structure is compatible with the model structure.
\end{definition}

Many of the model structures that we consider here are given by localization of a known model structure.  The idea is that we start with a model structure, then choose some set of morphisms that we would like to become weak equivalences.  Doing so typically forces many more morphisms to become weak equivalences than only those in the given set.  The theory of localizations of model categories, and conditions under which they exist, can be found in \cite{hirsch} or \cite{barwick}.

The definition of the localization of a model category makes use of \emph{homotopy mapping spaces}, which can be defined for any model category $\mathcal M$.  The idea is that, given two objects $X$ and $Y$ of a model category, one can define a simplicial set $\Map^h(X,Y)$ which behaves like the mapping space $\Map(X,Y)$ in a simplicial model category yet is homotopy invariant and defined even if $\mathcal M$ is not simplicial.

\begin{definition}
\begin{enumerate}
    \item Let $\mathcal M$ be a model category and $S$ a set of morphisms in $\mathcal M$.  A fibrant object $X$ of $\mathcal M$ is $S$-\emph{local} if 
    \[ \Map^h(B,X) \rightarrow \Map^h(A,X) \]
    is a weak equivalence of simplicial sets for every map $A \rightarrow B$ in $S$.
    
    \item A map $C \rightarrow D$ in $\mathcal M$ is an $S$-\emph{local equivalence} if
    \[ \Map^h(D,X) \rightarrow \Map^h(C,X) \]
    is a weak equivalence for every $S$-local object $X$.
\end{enumerate}
\end{definition}

All of the model structures we consider here satisfy the hypotheses of the following theorem, which we do not state in full detail.

\begin{theorem} \cite[4.1.1]{hirsch}
If $\mathcal M$ is a sufficiently nice model category and $S$ is a set of maps in $\mathcal M$, then there exists a model structure on the same underlying category of $\mathcal M$ in which the weak equivalences are the $S$-local equivalences, the cofibrations are those of $\mathcal M$, and the fibrant objects are the $S$-local objects.
\end{theorem}

Finally, we give a brief discussion of model structures on categories of diagrams of simplicial sets.

\begin{theorem} \cite[11.6.1]{hirsch}
Let $\mathcal C$ be a small category and $\SSets^\mathcal C$ the category of functors $\mathcal C \rightarrow \SSets$.  There is a model structure, called the \emph{projective model structure} on this category, in which the weak equivalences and fibrations $X \rightarrow Y$ are given by weak equivalences and fibrations of simplicial sets $X(c) \rightarrow Y(c)$ for every object $c$ of $\mathcal C$.
\end{theorem}

There is likewise an \emph{injective} model structure, where the weak equivalences and cofibrations are defined levelwise.  However, in all cases we consider here, this model structure coincides with the \emph{Reedy model structure}, which is often more convenient in practice.  We refer the reader to \cite{reedy} or \cite[15.3]{hirsch} for more details.  In this paper, we are interested in \emph{simplicial spaces}, or functors $\Deltaop \rightarrow \SSets$, as well as the variants $\Theta_n^{\op} \rightarrow \SSets$ which we introduce later in the paper. 

\section{$(\infty,0)$-categories}

Using our strategy given in the introduction, our first task is to give a concrete model for $(\infty,0)$-categories.  By definition, an $(\infty, 0)$-category should be a weak $\infty$-groupoid: a weak $\infty$-category in which every $i$-morphism, for every $i \geq 1$, is weakly invertible.  We make the following definition.

\begin{definition}
An $(\infty,0)$-\emph{category} is a topological space.
\end{definition}

Why is this definition sensible?  Given a topological space $X$, we can think of the points of $X$ as objects, and the paths between points as 1-morphisms. Then a homotopy between two paths with the same endpoints can be thought of as a 2-morphism, and we can continue to take homotopies between homotopies to make sense of $i$-morphisms for all $i \geq 1$.  Since paths and homotopies are invertible up to homotopy, we get a weak $\infty$-groupoid.

\begin{remark}
We have chosen one particular approach in this regard, following work such as \cite{arahh} or \cite{lurietft}.  A general principle in higher category theory, called the \emph{Homotopy Hypothesis}, is that weak $n$-groupoids model $n$-types, or topological spaces with nontrivial homotopy groups only in degrees $n$ or lower.  It stands to reason, then, that a weak $\infty$-groupoid should be a topological space.  We have chosen to take this principle so seriously that we take it as our definition.  One can just as well take a more categorical definition of what an $\infty$-groupoid should be and then try to prove the Homotopy Hypothesis for that particular definition; for example see \cite{bpsegal} or \cite{cisinski}.
\end{remark}

However, as in many situations in homotopy theory, it is preferable to work in the setting of simplicial sets rather than that of topological spaces.  The Quillen equivalence of Theorem \ref{ssetstop} tells us that topological spaces and simplicial sets have the same homotopy theory.  In practice, one might prefer one or the other, but from the point of view of homotopy theory they are equivalent.  Thus, we can also consider simplicial sets as models for $\infty$-groupoids.  However, it is preferable to restrict to the simplicial sets which are both fibrant and cofibrant; all objects are cofibrant, but it is really only the fibrant objects, the Kan complexes, which best model $\infty$-groupoids.

\begin{definition}
A simplicial set $K$ is a \emph{Kan complex} if a lift exists in any diagram of the form
\[ \xymatrix{V[n.k] \ar[d] \ar[r] & K \\
\Delta[n] \ar@{-->}[ur] } \]
where $n \geq 1$ and $0 \leq k \leq n$.
\end{definition}

The inclusions $V[n,k] \rightarrow \Delta[n]$ are called \emph{horn inclusions}.  To get an idea of what this lifting property means, let us look at the case when $n=2$.
The horn $V[2,0]$ can be depicted as
\[ \xymatrix{& 1  & \\
0 \ar[rr] \ar[ur] && 2} \] whereas $V[2,1]$ looks like
\[ \xymatrix{& 1  \ar[dr] & \\
0 \ar[ur] && 2} \] and finally $V[2,2]$ looks like
\[ \xymatrix{& 1  \ar[dr] & \\
0 \ar[rr] && 2.} \]
Having a lift with respect to $V[2,1]$ tells us that $K$ has composition, in the sense that any two 1-simplices, of which the source of one is the target of the other, can be filled to a 2-simplex; we think of the additional face as a composite of the original two 1-simplices.  However, this composition need not be unique.  In higher dimensions, having a lift with respect to the analogous \emph{inner horns}, for which $0<k<n$, gives information about composites for longer strings of 1-simplices, and about associativity of composition, at least up to homotopy.

The \emph{outer horns}, however, play a different role.  For $n=2$, the existence of lifts when $k=0$ and $k=2$ demonstrate the existence of left and right inverses to a given 1-simplex.  Thus, these lifts show that a Kan complex not only behaves like a category up to homotopy, but moreover like a groupoid up to homotopy.  This property agrees with our argument above that a Kan complex should model an $\infty$-groupoid.  Indeed, one can make sense of paths and homotopies and homotopies between homotopies, just as we do in a topological space, to think of a Kan complex as an $\infty$-groupoid.

Now that we have good ways to think about $(\infty,0)$-categories, we are ready to move up to $(\infty,1)$-categories.

\section{$(\infty,1)$-categories}

Following our principle that any $(\infty, n)$-category should be a category enriched in $(\infty,n-1)$-categories, we can take categories enriched in topological spaces or categories enriched in simplicial sets as a model for $(\infty,1)$-categories.

\begin{definition}
A \emph{simplicial category} is a category enriched in simplicial sets.  In other words, it has a collection of objects, together with, for any pair $(x,y)$ of objects, a simplical set $\Map(x,y)$, together with a compatible composition law.
\end{definition}

One can define topological categories analogously; we refer the reader to \cite{ilias} for the corresponding homotopy theory.

Since we want to look at each of our models homotopy-theoretically, we want to show that we have a good model structure for simplicial categories.  Let us first define the appropriate notion of weak equivalence, which can be thought of as a simplicial generalization of the definition of equivalence of categories.  Given a simplicial category, we denote by $\pi_0(\mathcal C)$ the category of components of $\mathcal C$, which has the same objects as $\mathcal C$ and in which
\[ \Hom_{\pi_0(\mathcal C)}(x,y) = \pi_0 \Map_\mathcal C(x,y). \]

\begin{definition}
A simplicial functor $f \colon \mathcal C \rightarrow \mathcal D$ is a \emph{Dwyer-Kan equivalence} if:
\begin{enumerate}
    \item for any $x,y \in \ob(\mathcal C)$, the induced map
    \[ \Map_\mathcal C(x,y) \rightarrow \Map_\mathcal D(fx,fy) \]
    is a weak equivalence of simplicial sets; and
    
    \item the induced functor $\pi_0(\mathcal C) \rightarrow \pi_0(\mathcal D)$ is essentially surjective.
\end{enumerate}
\end{definition}

\begin{theorem} \textup{\cite[1.1]{simpcat}} 
There is a model structure $\mathcal{SC}$ on the category
of small simplicial categories in which the weak equivalences are the Dwyer-Kan equivalences.
\end{theorem}

As for simplicial sets, we focus on the fibrant objects, which are precisely the simplicial categories whose mapping spaces are all Kan complexes.

However, there are good reasons to look for alternative models for $(\infty,1)$-categories.
\begin{itemize}
\item This model category does not satisfy good properties if we want to continue the process of enrichment to obtain models for $(\infty,2)$-categories.  The category of small categories enriched in small simplicial categories can be defined, but we cannot expect it to have a suitable model structure, since $\mathcal{SC}$ is not a cartesian model category.  

\item This model is too rigid to accommodate many examples.  The composition law in an enriched category is required to satisfy strict associativity and unitality, and we would like models for which these properties only hold up to homotopy.
\end{itemize}

Our discussion of Kan complexes earlier lends itself to one possible way to think of certain simplicial sets as $(\infty,1)$-categories.  We can retain the conditions which encode category-like behavior but exclude the ones which impose the existence of inverses.

\begin{definition}
A simplicial set $K$ is a \emph{quasi-category} if a lift exists in any diagram of the form
\[ \xymatrix{V[n.k] \ar[d] \ar[r] & K \\
\Delta[n] \ar@{-->}[ur] } \]
where $n \geq 1$ and $0 < k < n$.
\end{definition}

\begin{theorem} \cite[2.13]{dugspir}, \cite{joyal}, \cite[2.2.5.1]{lurie}
There is a cartesian model structure $\qcat$ on the category of simplicial sets in which the fibrant objects are the quasi-categories.
\end{theorem}

To show that quasi-categories provide a good model for $(\infty,1)$-categories, it suffices to have a model category is Quillen equivalence between $\Qcat$ and $\mathcal{SC}$.  We first need to define an adjoint pair of functors between the underlying categories.  The following definition was first given by Cordier and Porter \cite{cp}.

\begin{definition}
	The \emph{coherent nerve functor} $\widetilde{N} \colon \SSets \rightarrow \mathcal{SC}$ is defined by 
	\[ \widetilde{N}(\mathcal C)_n = \Hom_{\mathcal{SC}}(F_*[n], \mathcal C), \] 
	where $F_*[n]$ denotes a simplicial resolution of the category $[n]$.
\end{definition}

Different approaches to the following result can be found in \cite{joyal} and \cite{lurie}.

\begin{prop} 
	The coherent nerve functor $\widetilde{N}$ admits a left adjoint $\mathfrak C$.
\end{prop}

This adjoint pair gives us our desired means of comparison.

\begin{theorem} \cite[8.2]{dugspir}, \cite{joyal}, \cite[2.2.5.1]{lurie}
The adjoint pair 
\[ \mathfrak C \colon \qcat\leftrightarrows \mathcal{SC} \colon \widetilde{N} \]
is a Quillen equivalence.
\end{theorem}

However, there are other approaches to defining models for $(\infty,1)$-categories with weak composition whose starting point is instead the simplicial nerve functor, which takes a simplicial category to a simplicial space, or bisimplicial set, $\Deltaop \rightarrow \SSets$.  To define this functor, it is convenient to observe that a simplicial category, in our sense, can be thought of as a simplicial object $\Deltaop \rightarrow \Cat$ for which the face and degeneracy maps are all the identity on objects.

\begin{definition} 
	Let $\mathcal C$ be a simplicial category, thought of as a functor $\Deltaop \rightarrow \Cat$.  Its \emph{simplicial nerve} is the simplicial space $\snerve (\mathcal C)$ defined by 
	\[ \snerve(\mathcal C)_{*,m} = \nerve(\mathcal C_m). \]
\end{definition}

We want to look at simplicial spaces that arise as simplicial nerves of simplicial categories and identify what properties they must have.  The first thing to observe is that, since simplicial categories do not have a simplicial structure on their objects, the simplicial set $\snerve(\mathcal C)_0$ must be discrete.  We thus make the following definition.

\begin{definition}
A \emph{Segal precategory} is a simplicial space $X$ such that $X_0$ is discrete.
\end{definition}

We denote the category of Segal precategories by $\SSets^{\Deltaop}_{\disc}$.

More interestingly, however, is the structure that we get from the composition of mapping spaces in a simplcial category.  To describe it, we need to set up some notation.

In the category $\Delta$, consider the maps $\alpha^i \colon [1] \rightarrow [k]$, where $0\leq i <k$, given by $\alpha^i(0)=i$ and $\alpha^i(1)=i+1$.   Define the simplicial set 
\[ G(k)= \bigcup_{i=0}^{k-1} \alpha^i \Delta[1] \subseteq \Delta[k]. \]
Alternatively, we can write 
\[ G(k) = \underbrace{\Delta[1] \amalg_{\Delta[0]} \cdots \amalg_{\Delta[0]} \Delta[1]}_k \]
where the right-hand side is colimit of representables induced by the diagram
\[ [1] \overset{d^0}{\rightarrow} [0] \overset{d^1}{\leftarrow} \cdots \overset{d^0}{\rightarrow} [0] \overset{d^1}{\leftarrow} [1] \]
in the category $\Delta$.

Since we are working with simplicial spaces, rather than simplicial sets, we want to think of $G(k)$ and $\Delta[k]$ in that context.  There are two ways to think of a simplicial set $K$ as a simplicial space: as a constant simplicial diagram given by $K$, or as a diagram of discrete simplicial sets given by the simplices of $K$.  Since the former is typically still denoted by $K$, we denote the latter by $K^t$; we think of it as the ``transpose" of the constant diagram, which constant in the other simplicial direction.  Thus, we have $K^t_k = K_k$, where the right-hand side is a constant simplicial set on the set $K_k$.
 
With this notation in place, let us consider the inclusion of simplicial spaces $G(k)^t \rightarrow \Delta[k]^t$.  

\begin{definition}
Given any simplicial space $W$ and any $k \geq 2$, the \emph{Segal map} is the induced map
\[ \Map(\Delta[k]^t, W) \rightarrow \Map(G(k)^t, W) \]
which can be rewritten simply as
\[ W_k \rightarrow \underbrace{W_1 \times_{W_0} \cdots \times_{W_0} W_1}_k. \]
\end{definition}

Now it is not hard to check the following characterization of simplicial nerves.

\begin{prop}
Let $X$ be a simplicial space which can be obtained as the nerve of a simplicial category.  Then $X$ is a Segal precategory and, for every $k \geq 2$, the Segal maps
\[ X_k \rightarrow \underbrace{X_1 \times_{X_0} \cdots \times_{X_0} X_1}_k \]
are isomorphisms of simplicial sets.
\end{prop}

Since we want a model for $(\infty,1)$-categories which is less rigid than that of simplicial categories, we can relax the condition that the Segal maps be isomorphisms.  We thus make the following definition.

\begin{definition}
	A \emph{Segal space} is a Reedy fibrant simplicial space $W$ such that the Segal maps are weak equivalences of simplicial sets for all $k \geq 2$.
\end{definition}

This requirement that the Segal maps be weak equivalences is often referred to as the \emph{Segal condition}.

\begin{theorem} \cite[7.1]{rezk}
	There is a model structure, which we denote by $\sesp$, on the category of simplical spaces such that all objects are cofibrant and the fibrant objects are precisely the Segal spaces.
\end{theorem}

Segal spaces, with no further assumptions, do not quite model $(\infty,1)$ categories.  While the Segal condition allows us to define an up-to-homotopy composition, we have a space, rather than a set, of objects.  In other words, Segal spaces model categories internal to spaces, rather than enriched in spaces.  There are two approaches to remedying this difficulty.

For our first model, taking the output of the simplicial nerve as a guide, we retain the discreteness of the space in degree 0.  The following definition first appeared in \cite{dks}.

\begin{definition}
A \emph{Segal category} is a Segal precategory for which the Segal maps are weak equivalences for all $k \geq 2$.
\end{definition}

To show that Segal categories do indeed give a model for $(\infty,1)$-categories, we need to define a model structure for them and show that it is Quillen equivalent to $\mathcal{SC}$.  We first need a sensible notion of weak equivalence, and again we use simplicial categories as a guide.

We can apply much of the language of simplicial categories in the context of Segal categories.  Given a Segal category $X$, we refer to the discrete space $X_0$ as its \emph{set of objects}.  We define \emph{mapping spaces} between objects $x$ and $y$ as the homotopy pullback
\[ \xymatrix{\map_X(x,y) \ar[d] \ar[r] & X_1 \ar[d] \\
\{(x,y)\} \ar[r] & X_0 \times X_0.} \] 
Using the fact that the Segal maps are weak equivalences, there is a notion of composition of mapping spaces, but it is only defined up to homotopy \cite[5.3]{rezk}. Thus, we get a desired ``weak composition" compared to the stricter composition in a simplicial category.  Taking the objects and the sets of path components of the mapping spaces, we obtain an ordinary category $\Ho(X)$ associated to a Segal category $X$.  

Although we do not go into detail here, there is a suitable functor $L$ which takes any Segal precategory to a Segal category which is weakly equivalent to it in the model category $\sesp$ \cite[\S 5]{thesis}.  Thus, for more general Segal precategories, we can first apply the functor $L$ and then apply the definitions just described.  In particular, we make the following definition.

\begin{definition}
	A map $f \colon X \rightarrow Y$ of Segal precategories is a \emph{Dwyer-Kan equivalence} if:
	\begin{enumerate}
		\item for any objects $x,y \in X_0$, the induced map \[\map_{LX}(x,y) \rightarrow \map_{LY}(fx,fy) \]
		is a weak equivalence of simplicial sets, and
		
		\item the induced map on homotopy categories $\Ho(LX) \rightarrow \Ho(LY)$ is essentially surjective.
	\end{enumerate}
\end{definition}

\begin{theorem} \cite[5.1, 7.1, 7.5]{thesis}, \cite{pel}
There are two model structures on the category of Segal precategories, each of which has Dwyer-Kan equivalences as weak equivalences.
\begin{enumerate}
	\item The first model structure, which we denote by $\Secat_c$, has all objects cofibrant and fibrant objects precisely the Reedy fibrant Segal categories, and this model structure is cartesian.
	
	\item The second model structure, which we denote by $\Secat_f$, has fibrant objects precisely the projective fibrant Segal categories.  The cofibrant objects are closely related to cofibrant objects in the projective model structure on simplicial spaces.
	
	\item The two model structures are Quillen equivalent via the identity functors:
	\[ \id \colon \Secat_f \rightleftarrows \Secat_c \colon \id. \]
\end{enumerate}
\end{theorem}

\begin{remark}
   The astute reader might have noticed the following incongruity in our definitions.  We assume that a Segal space is Reedy fibrant, but we make no such assumption on a Segal category.  In particular, what we'd like to say is that a Segal category is simply a Segal space with 0-space discrete.  That point of view works nicely if all we wanted was the model structure $\Secat_c$.  Indeed, this model structure is preferable for many purposes and was the one originally developed by Pelissier in \cite{pel}.  
   
   Unfortunately, there is no direct Quillen equivalence between $\mathcal{SC}$ and $\Secat_c$, essentially because there are too many cofibrations in $\Secat_c$ compared to $\mathcal{SC}$.  The model structure $\Secat_f$ is designed to facilitate this comparison.
   
   We could instead drop the Reedy fibrancy condition on Segal spaces (and many authors do), but then the face maps used to define the limits in the codomains of the Segal maps need not be fibrations, so we need to take a homotopy limit instead.  For Segal categories, the discreteness in degree zero allows us to consider strict limits, so we do not need Reedy fibrancy. One could take an analogous Segal space localization in the projective model structure as well.  There are reasons why this model structure is not as well-behaved for comparisons as the one we have chosen; see \cite[\S 7]{thesis} for further discussion on this point.
\end{remark}

\begin{theorem} \cite[8.6]{thesis}
	The simplicial nerve functor $\mathcal{SC} \rightarrow \SSets^{\Deltaop}_{\disc}$ has a left adjoint which we denote by $F$.  This adjoint pair induces a Quillen equivalence
	\[ F \colon \Secat_f \rightleftarrows \mathcal{SC} \colon \snerve. \]
\end{theorem}

This left adjoint functor $F$ can be thought of as a ``rigidification" of a Segal category to a simplicial category.

The model structure $\Secat_c$, on the other hand, is well-suited to comparison with our alternate approach to making Segal spaces models for $(\infty,1)$-categories.  

A first question we might ask is why we want an alternative to Segal categories.  The main difficulty with them is the fact that we need their degree 0 space to be discrete, which is an unnnatural condition from the perspective of homotopy theory.  We could weaken this condition to homotopy discreteness, but there is another point of view, which we now describe.

Let us return to the way in which we talked about Segal categories in the language of simplicial categories.  The definitions we made above make sense for more general Segal spaces.  In particular, given a Segal space $W$, we define its \emph{space of homotopy equivalences} to be the subspace of $W_1$ whose image in $\Ho(W)$ consists of isomorphisms, and we denote it by $W_{\heq}$.  Observe that the degeneracy map $s_0 \colon W_0 \rightarrow W_1$ factors through $W_{\heq}$.

\begin{definition}
A Segal space $W$ is \emph{complete} if the map $W_0 \rightarrow W_{\heq}$ is a weak equivalence of simplicial sets.
\end{definition}

The idea behind this completeness condition is that the space of objects, which we no longer assume to be discrete, is instead encoded into the space of morphisms.

\begin{theorem} \cite[7.2]{rezk}
There is a model structure on the category of simplicial spaces, denoted by $\css$, in which all objects are cofibrant and the fibrant objects are precisely the complete Segal spaces.  Furthermore, this model structure is cartesian.
\end{theorem}

To compare this model structure to $\Secat_c$, we need a way to ``discretize" the degree zero space of a complete Segal space to get a Segal category.

\begin{theorem} \cite[6.3]{thesis}
The inclusion functor $I \colon \SSets^{\Deltaop}_{\disc} \rightarrow \SSets^{\Deltaop}$ admits a right adjoint $D$.  This adjoint pair induces a Quillen equivalence
\[ I \colon \Secat_c \leftrightarrows \css \colon D.\]
\end{theorem}

To understand why this theorem works, let us look at the role of Dwyer-Kan equivalences in the model structure $\css$ via the following theorem of Rezk.

\begin{theorem}  \cite[7.6, 7.7]{rezk}
\begin{enumerate}
\item Let $f \colon W \rightarrow Z$ be a map of Segal spaces.  Then $f$ is a Dwyer-Kan equivalence if and only if it is a weak equivalence in $\css$.

\item Let $f \colon W \rightarrow Z$ be a map of complete Segal spaces.  Then $f$ is a Dwyer-Kan equivalence if and only if it is a weak equivalence in $\css$.
\end{enumerate}
\end{theorem}

The right adjoint $D$ to the inclusion functor effectively collapses the simplicial set in degree zero to its set of components.  If we apply this functor to a complete Segal space $W$, then the result is typically no longer complete, but it is a Segal category.  If $W \rightarrow Z$ is a weak equivalence between complete Segal spaces in $\css$, the Quillen equivalence above essentially reduces to showing that the functor $D$ preserves Dwyer-Kan equivalences.

\begin{remark} \label{cssandsecat}
    A natural question to ask is when complete Segal spaces and Segal categories coincide.  The answer is not very often!  A simplicial space $X$ which is both a complete Segal space and a Segal category satisfies both $X_0 \simeq X_{\heq}$ and $X_0$ is discrete.  In other words, the space of homotopy equivalences of $X$ must be homotopy discrete, and $\Ho(X)$ is a category with no non-identity automorphisms of objects.  It can have non-identity isomorphisms, but they must be unique between two given objects.  An analogous structure is a simplicial category for which all homotopy automorphisms are homotopic to identity morphisms and whose homotopy equivalences between two given objects form a contractible space.
\end{remark}

That there are also Quillen equivalences (in both directions!) between $\Qcat$ and $\css$, and between $\Qcat$ and $\Secat_c$, was proved by Joyal and Tierney \cite{jt}.  We discuss one of these comparisons.

\begin{theorem} \cite[4.11]{jt}
The evaluation map $\ev_0$, taking a simplicial space $W$ to the simplicial space $W_{*,0}$, is right adjoint to the inclusion functor $i \colon \SSets \rightarrow \SSets^{\Deltaop}$, taking a simplicial set $K$ to the simplicial space $Z$ with $Z_{*,n} = K$ for all $n$.  This adjoint pair induces a Quillen equivalence 
\[ i \colon \Qcat \rightleftarrows \css \colon \ev_0. \] 
\end{theorem}

While we have by no means covered all possible models for $(\infty,1)$-categories, the ones we have here give a sense of how they can be described.  We now look toward moving up one more level to $(\infty,2)$-categories.

\section{$(\infty,2)$-categories as enriched categories}

The first approach to obtaining models for $(\infty,2)$-categories is to take categories enriched in any one of our models for $(\infty,1)$-categories.  Simply defining such objects is not a problem for any of the models that we have, as each of the underlying categories has a well-behaved monoidal structure under cartesian product.  However, as we have already seen for simplicial categories, we do not expect that all of these models have corresponding model structures.  The key feature we need if we want such a model structure is that the enrichment is taken over a cartesian model category.  As we saw in the previous section, three of the models have cartesian model structures: $\Qcat$, $\css$, and $\Secat_c$.

In fact, we do get model structures if we enrich in any of these model structures.  The general strategy is spelled out by Lurie in in Appendix A of \cite{lurie}.  The main idea is that we want a model structure on the category of small categories enriched in a cartesian model category $\mathcal V$, denoted by $\mathcal V$-$\Cat$, which is analogous to the model structure for simplicial categories, with a variant of Dwyer-Kan equivalences as weak equivalences.  But how do we define these maps in a more general enriched category?

The first condition, that of being homotopically fully faithful, is not hard to generalize.  We simply ask that the induced maps on mapping objects, which we denote by $\uMap_\mathcal C(x,y)$, be weak equivalences in $\mathcal V$.  But what about essential surjectivity?  For Dwyer-Kan equivalences of simplicial categories, we used the category of components $\pi_0 \mathcal C$ of a simplicial category $\mathcal C$.  Since there we had mapping simplicial sets, taking $\pi_0$ was a natural thing to do.  More generally, we define $\pi_0 \mathcal C$ to be the category whose objects are those of $\mathcal C$ and whose morphisms are given by
\[ \Hom_{\pi_0 \mathcal C}(x,y) = \Hom_{\Ho(\mathcal V)}(\ast, \uMap_\mathcal C(x,y)), \]
where $\ast$ denotes the terminal object of $\mathcal V$.

\begin{definition}
A $\mathcal V$-enriched functor $f \colon \mathcal C \rightarrow \mathcal V$ is a \emph{Dwyer-Kan equivalence} if:
\begin{enumerate}
    \item for every $x,y \in \ob(\mathcal C)$, the map
    \[ \uMap_\mathcal C(x,y) \rightarrow \uMap_\mathcal D(fx,fy) \]
    is a weak equivalence in $\mathcal V$, and 
    
    \item the functor $\pi_0 (\mathcal C) \rightarrow \pi_0(\mathcal D)$ is essentially surjective.
\end{enumerate}
\end{definition}

Let us now show that we have the desired model categories.  In the case in which we enrich in the complete Segal space model structure $\css$, a full proof applying this strategy is given by the author and Rezk in \cite{inftyn1}.

\begin{theorem} \cite[3.11]{inftyn1}
There is a cofibrantly generated model structure on $\css$-$\Cat$ in 
which the weak equivalences $f \colon \mathcal C \rightarrow \mathcal D$ are Dwyer-Kan equivalences.
\end{theorem}

Essentially the same proof technique can be used to obtain a similar result for enriching in $\Secat_c$.

\begin{theorem} 
There is a cofibrantly generated model structure on $\Secat_c$-$\Cat$ in which the weak equivalences $f \colon \mathcal C \rightarrow \mathcal D$ are the Dwyer-Kan equivalences.
\end{theorem}

The corresponding theorem for enriching in $\Qcat$ is implicit in \cite{lurie}; we give a brief sketch of the proof here.

\begin{theorem} 
There is a cofibrantly generated model structure on $\Qcat$-$\Cat$ in which the weak equivalences $f \colon \mathcal C \rightarrow \mathcal D$ are the Dwyer-Kan equivalences.
\end{theorem}

\begin{proof}
We apply the criteria of \cite[A.3.2.4]{lurie}.  The only condition that is not straightforward to check is that weak equivalences in $\Qcat$ preserve filtered colimits.  However, this result is proved in \cite[2.13]{dugspims}.
\end{proof}

Given these model structures, we would like to know that they are all equivalent to one another, since we are enriching in Quillen equivalent model categories. The following result is a consequence of \cite[A.3.2.6]{lurie}.

\begin{theorem}
There are Quillen equivalences
\[ \xymatrix{& \Qcat-\Cat \ar[dl] \ar[dr] & \\
\Secat_c-\Cat \ar[rr] \ar[ur] && \css-\Cat. \ar[ul] \ar[ll]}\]
\end{theorem}

However, if we want to move to higher $(\infty,n)$-categories by iterating this process, we have the same trouble again, as the model categories we have given here are not cartesian. 

Furthermore, if we want to accommodate more flexible examples, we want to have models with less rigid composition structure.  There are a number of approaches which generalize the models for $(\infty,1)$-categories in different ways.  In the next few sections, we look at some of the possibilities, starting with generalizations of complete Segal spaces and Segal categories.

\section{Multisimplicial models for $(\infty,2)$-categories}

When trying to move to $(\infty,2)$-categories, the first thought one might have is to add a higher categorical level by adding another simplicial level.  We see this kind of intuition when we generalize from the nerves of categories (simplicial sets) to simplicial nerves of simplicial categories (bisimplicial sets).  

When we consider Segal categories and complete Segal spaces, our starting point is the notion of Segal space.  Thus, it is natural to start with a structure that satisfies Segal conditions in two simplicial directions.  There are different ways to describe such a structure, but we begin with the following description.

\begin{definition}
A Reedy fibrant functor $W \colon \Deltaop \times \Deltaop \rightarrow \SSets$ is a \emph{double Segal space} if the Segal maps
\[ W_{k,*} \rightarrow \underbrace{W_{1,*} \times_{W_{0,*}} \cdots \times_{W_{0,*}} W_{1,*}}_k \]
and
\[ W_{*,k} \rightarrow \underbrace{W_{*,1} \times_{W_{*,0}} \cdots \times_{W_{*,0}} W_{*,1}}_k \]
are weak equivalences for any $k \geq 2$.  In other words, the simplicial spaces $W_{k,*}$ and $W_{k,*}$ are Segal spaces for any fixed $k$.  
\end{definition}

The definition we have given treats the two simplicial directions equally, but it can be convenient to distinguish the two in the following way.

\begin{definition} \label{segalobj}
A Reedy fibrant functor $W \colon \Deltaop \rightarrow \SSets^{\Deltaop}$ is a \emph{Segal object in Segal spaces} if the maps
\[ W_k \rightarrow \underbrace{W_1 \times_{W_0} \cdots \times_{W_0} W_1}_k \]
are weak equivalences in $\sesp$ for all $k \geq 0$.
\end{definition}

Indeed, one can check that these two definitions agree.

\begin{prop}
A bisimplicial space is a double Segal space if and only if it is a Segal object in Segal spaces.
\end{prop}

The following model structure appears in \cite{boors3}, although it was almost certainly known to experts previously.

\begin{theorem} 
There is a model structure on the category of bisimplicial spaces in which the fibrant objects are precisely the double Segal spaces.
\end{theorem}

Just as Segal spaces do not quite model $(\infty,1)$-categories, we do not expect double Segal spaces to model $(\infty,2)$-categories.  Recall that the difference for Segal spaces was that they give a model for categories internal to spaces, rather than categories enriched in spaces.  In other words, we have a space of objects as well as of morphisms.  The problem here is similar.  A double Segal space encodes the information of a homotopical double category.  A \emph{double category} is a category internal to categories, and as such has objects, horizontal morphisms, vertical morphisms, and squares.  

More precisely, suppose that $W$ is a double Segal space.  We can think of $X_{0,0}$ as the space of objects, $W_{0,1}$ as the space of vertical morphisms, $W_{1,0}$ as the space of horizontal morphisms, and $W_{1,1}$ as the space of squares.  In other words, the 2-morphisms are encoded in squares like the following:
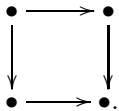
\[ \xymatrix{\bullet \ar[r] \ar[d] & \bullet \ar[d] \\
\bullet \ar[r] & \bullet.} \]
However, when we look at $(\infty,2)$-categories, we typically want to think of 2-morphisms as being of the form
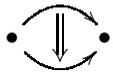
\[\xymatrix{{\bullet} \ar@/^1pc/[r]^{}="30" \ar@/_1pc/[r]^{}="31" \ar@{=>}"30";"31"
& {\bullet.}
}\]
In particular, do not want to have nontrivial vertical morphisms, but only horizontal ones, and thus rather than configurations of squares we want what are often called ``globular" diagrams.  To model such a structure without interesting vertical morphisms, we want to ask that the simplicial space $W_{0,*}$ be essentially constant.  As in the case of $(\infty,1)$-categories, we also want to ask either that $W_{0,0}$, the space of objects, be discrete, or to have a corresponding completeness condition.  Additionally, since we want to think of $(\infty,2)$-categories as enriched in $(\infty,1)$-categories, we want the same kind of condition on the 1-morphisms: either that the space of such be discrete, or that we have a corresponding completeness condition.  The focus of this section, then, is to describe how to impose these kinds of conditions appropriately.

Let us first consider the option of imposing two completeness conditions.  As a first step, let us define Segal objects in complete Segal spaces, building on Definition \ref{segalobj}.

\begin{definition}
A Reedy fibrant functor $W \colon \Deltaop \rightarrow \SSets^{\Deltaop}$ is a \emph{Segal object in complete Segal spaces} if, for any $k \geq 2$, the Segal map
\[ W_k \rightarrow \underbrace{W_1 \times_{W_0} \cdots \times_{W_0} W_1}_k \]
is a weak equivalence in the complete Segal space model structure $\css$.
\end{definition}

This definition, as we have given it, is quite formal, so let us investigate the structure further.  We first make the following definition of mapping objects, generalizing mapping spaces in a Segal space.

\begin{definition}
Let $W$ be a double Segal space.  Then for any $x,y \in W_{0,0,0}$, the \emph{mapping object} $\umap_W(x,y)$ is defined to be the simplicial space defined as the pullback
\begin{equation}  \label{mappingobj}
\xymatrix{\umap_W(x,y) \ar[r] \ar[d] & W_{1,*} \ar[d]^{(d_1, d_0)} \\
\{(x,y)\} \ar[r] & W_{0,*} \times W_{0,*}. } 
\end{equation}
\end{definition}

As for mapping spaces in Segal spaces, the fact that $W$ is Reedy fibrant implies that this pullback is actually a homotopy pullback.  We can use these mapping objects in the following alternative characterization of Segal objects in complete Segal spaces.

\begin{prop}
A Reedy fibrant functor $W \colon \Deltaop \times \Deltaop \rightarrow \SSets$ is a Segal object in complete Segal spaces if and only if it satisfies the following conditions:
\begin{enumerate}
\item \label{cssobj1}
for any $m \geq 2$, the Segal map
\[ W_{m,*} \rightarrow \underbrace{W_{1,*} \times_{W_{*,0}} \cdots \times_{W_{*,0}} W_{1,*}}_m \]
is a weak equivalence in $\css$; and

\item \label{cssobj2}
for every $x, y \in W_{0,0,0}$, $\umap_W(x,y)$ is a complete Segal space.
\end{enumerate}
\end{prop}

\begin{proof}
It is not hard to check that a bisimplicial space satisfying these two conditions is a Segal object in complete Segal spaces.

Conversely, suppose that $W$ is a Segal object in complete Segal spaces.  Then \eqref{cssobj1} holds, since these maps are assumed to be weak equivalences of complete Segal spaces.  

To check \eqref{cssobj2}, consider $\umap_W(x,y)$ for fixed $x,y \in W_{0,0,0}$.  Since $W$ is assumed to be Reedy fibrant, the right vertical map in \eqref{mappingobj} is a fibration between complete Segal spaces, which are the fibrant objects in $\css$.  Since the discrete object $\{(x,y)\}$ is also a fibrant object in $\css$, the pullback must be as well.  It follows that $\umap_W(x,y)$ is fibrant, namely, a complete Segal space.
\end{proof}

We use the approach of this second characterization to define complete Segal objects.

\begin{definition} \label{2-foldcss}
A Reedy fibrant functor $W \colon \Deltaop \times \Deltaop \rightarrow \SSets$ is a \emph{complete Segal object in complete Segal spaces}, or a \emph{2-fold complete Segal space} if:
\begin{enumerate}
\item \label{2foldcss1}
for every $k \geq 0$, the simplicial space $W_{*,k}$ is a complete Segal space;

\item \label{2foldcss2}
for every $x,y \in W_{0,0,0}$, the simplicial space $\umap_W(x,y)$ is a complete Segal space; and

\item \label{2foldcss3}
the simplicial space $W_{0,*}$ is essentially constant.
\end{enumerate}
\end{definition}

\begin{remark}
	The definition of 2-fold complete Segal space is often stated with \eqref{2foldcss1} replaced by the the seemingly weaker condition that each $W_{*,k}$ be a Segal space with only $W_{*,0}$ assumed to be complete.  However, as proved by Johnson-Freyd and Scheimbauer \cite[2.8]{jfs}, these conditions actually imply condition \eqref{2foldcss1} as we have stated it.
\end{remark}

The following model structure, like the one for complete Segal spaces, is obtained via localization of the Reedy model structure on bisimplicial spaces.

\begin{theorem} \cite{inftyn2}
There is a model structure $\css(\css)$ on the category of bisimplical spaces in which the fibrant objects are precisely the 2-fold complete Segal spaces.
\end{theorem}

Because they play such a critical role in the theory of complete Segal spaces and their comparison with other models, and are used to define the enriched category models for $(\infty,2)$-categories, let us look at Dwyer-Kan equivalences in this setting.  We have defined mapping objects, so we can define homotopical fully faithfulness, but for essential surjectivity we need a notion of homotopy category of a 2-fold complete Segal space.  Since one might think more naturally of a homotopy 2-category in this context, we need to reduce a categorical level, which we do by defining the underlying complete Segal space of a 2-fold complete Segal space.

\begin{definition}
Let $\tau_\Delta \colon \Delta \rightarrow \Delta \times \Delta$ be the functor defined by $[m] \mapsto ([m], [0])$.  The induced functor $\tau_\Delta^* \colon \SSets^{\Deltaop \times \Deltaop} \rightarrow \SSets^{\Deltaop}$ defines the \emph{underlying simplicial space} of a bisimplicial space.
\end{definition}

\begin{prop} \cite{inftyn2}
If $W$ is a 2-fold (complete) Segal space, then its underlying simplicial space $\tau_\Delta^*W$ is a (complete) Segal space.
\end{prop}

In particular, $\tau_\Delta^*W$ has an associated homotopy category.  In the following definition, we let $L$ denote the functorial fibrant replacement functor in the double Segal space model structure.

\begin{definition} 
Let $W$ and $Z$ be objects of $\SSets^{\Deltaop \times \Deltaop}$.  A functor $f \colon W \rightarrow Z$ is a \emph{Dwyer-Kan equivalence} if:
\begin{itemize}
    \item for any $x,y \in W_{0,0,0}$, the induced map $\umap_{LW}(x,y) \rightarrow \umap_{LZ}(fx,fy)$ is a weak equivalence in $\css$, and 
    
    \item the induced map $\Ho(\tau_\Delta^* LW) \rightarrow \Ho(\tau_\Delta^*LZ)$ is essentially surjective.
\end{itemize}
\end{definition}

\begin{theorem} \cite{inftyn2}
A map $f \colon W \rightarrow Z$ of Segal objects in $\css$ is a Dwyer-Kan equivalence if and only if it is a weak equivalence in $\css(\css)$.
\end{theorem}

To show that 2-fold complete Segal spaces give a good model for $(\infty, 2)$-categories, we would like to show that the model category $\css(\css)$ is Quillen equivalent to $\css$-$\Cat$.  Just as in the comparison between simplicial categories and complete Segal spaces, we make use of an intermediate model structure which uses a discreteness condition.  Thus, we turn to one of our variant models.

\begin{definition}
A \emph{Segal category object in complete Segal spaces} is a functor $W \colon \Deltaop \rightarrow \SSets^{\Deltaop}$ such that $W_0$ is discrete and for each $k \geq 0$ the map 
\[ W_k \rightarrow \underbrace{W_1 \times_{W_0} \cdots \times_{W_0} W_1}_k \]
is a weak equivalence of complete Segal spaces.
\end{definition}

As for Segal categories, observe that we are not simply saying that a Segal category object is a Segal object in complete Segal spaces with the appropriate discreteness, because we do not want to impose Reedy fibrancy.  We once again make use of two different model structures for our desired comparison.

\begin{theorem} \cite{inftyn1}
There are two model structures on the category of functors $W \colon \Deltaop \rightarrow \SSets^{\Deltaop}$ such that the fibrant objects are Segal category objects and the weak equivalences are the Dwyer-Kan equivalences:
\begin{itemize}
    \item the model structure $\Secat_c(\css)$, in which the fibrant objects are the Segal category objects in complete Segal spaces which are Reedy fibrant, and
    
    \item the model structure $\Secat_f(\css)$, in which the fibrant objects are the Segal category objects in complete Segal sapces which are projective fibrant.
\end{itemize}
Furthermore, there is a Quillen equivalence
\[ \Secat_f(\css) \rightleftarrows \Secat_c(\css) \]
given by the identity functors.
\end{theorem}

Now, the following result mirrors the Quillen equivalences between simplicial categories, Segal categories, and complete Segal spaces.

\begin{theorem} \cite{inftyn1}, \cite{inftyn2}
\begin{enumerate}
    \item The bisimplicial nerve functor induces a Quillen equivalence
    \[ \css-\Cat \leftrightarrows \Secat_f(\css). \]
    
    \item The inclusion functor induces a Quillen equivalence
    \[ \Secat_c(\css) \rightleftarrows \css(\css). \]
\end{enumerate}
\end{theorem}

This theorem brings together all the models we have discussed thus far, establishing that they are all equivalent to one another.  While the technical points of this proof are quite difficult, the idea behind it is to apply the same methods as we did in the case of $(\infty,1)$-categories, but replacing the model structure on simplicial sets by $\css$.  

However, we have only considered two of the four possible combinations of completeness and discreteness conditions on double Segal spaces.  In particular, one could argue that we have been giving preferential treatment to the completeness condition.  Historically, however, a notion of Segal category objects in Segal categories came first, in work of Hirschowitz and Simpson \cite{hs}, and are treated in \cite{simpson}.

\begin{definition}
A bisimplicial space $X \colon \Deltaop \times \Deltaop \rightarrow \SSets$ is a \emph{Segal 2-precategory} if:
\begin{enumerate}
\item for any $m \geq 0$, the simplicial set $X_{m,0}$ is discrete; in other words, $X$ can be thought of as a simplicial object in Segal precategories; and

\item for any $k \geq 0$, the simplicial set $X_{0,k}$ is discrete.
\end{enumerate}
If, in addition, for any $k \geq 2$, the Segal map
\[ X_{k,*} \rightarrow \underbrace{X_{1, *} \times_{X_{0,*}} \cdots \times_{X_{0,*}} X_{1,*}}_k \]
is a Dwyer-Kan equivalence in $\Secat_c$, then $X$ is a \emph{Segal 2-category}.
\end{definition}

In other words, if we think of a Segal 2-category as a functor $X \colon \Deltaop \rightarrow \SSets^{\Deltaop}_{\disc}$, then each $X_k$ is a Segal category, so for any $m \geq 2$, the Segal map
\[ X_{k,m} \rightarrow \underbrace{X_{k,1} \times_{X_{k,0}} \cdots \times_{X_{k,0}} X_{k,1}}_m \]
is a weak equivalence of simplicial sets.  

\begin{theorem} \cite{pel}, \cite{discrete}
There is a model structure on the category of Segal 2-precategories such that the weak equivalences are Dwyer-Kan equivalences and the fibrant objects are the Reedy fibrant Segal 2-categories.
\end{theorem}

Once again, we could take an analogous structure which is projective fibrant, rather than Reedy (or even mix and match the two in the two simplicial directions!), but we do not worry about this structure here, since their main purpose was to aid in comparison with an enriched category model, and we do not need to make such a comparison here.

Now we would like to know that all of these models are equivalent.  The idea is that we can use the Quillen equivalence between $\Secat_c$ and $\css$ to do so.  Once again, the proofs are more delicate, but the core idea comes from that original comparison.

\begin{theorem} \cite{discrete}
The inclusion functors and their right adjoints, which are given by appropriate discretization functors, induce Quillen equivalences
\[ \Secat_c(\Secat_c) \rightleftarrows \Secat_c(\css) \rightleftarrows \css(\css). \]
\end{theorem}

We can summarize the results of this section via the following diagram, in which the displayed arrows are the left adjoint of a Quillen equivalence:
\[ \xymatrix{\css-\Cat & \Secat_f(\css) \ar[r] \ar[l] & \Secat_c(\css) \ar[r] & \css(\css) \\
&& \Secat_c(\Secat_c) \ar[u] & } \]

\begin{remark}
In theory, there is one remaining possibility, that of complete Segal objects in Segal categories.  However, although such objects can be defined, they are not substantively different from Segal 2-categories.  The problem goes back to the question of when Segal categories are complete Segal spaces. 

To see what happens, let us modify Definition \ref{2-foldcss} to what we expect a complete Segal object in Segal categories to be.  Such an object should be a functor $W \colon \Deltaop \rightarrow \SSets^{\Deltaop}_{\disc}$, or a functor $W \colon \Deltaop \times \Deltaop \rightarrow \SSets$ with each $W_{m,0}$ discrete. It should satisfy:
\begin{enumerate}
\item \label{cssinsecat1}
for every $k \geq 0$, the simplicial space $W_{*,k}$ is a complete Segal space;

\item for every $x,y \in W_{0,0,0}$, the simplicial space $\umap_W(x,y)$ is a Segal category; and

\item \label{cssinsecat3} the simplicial space $W_{0,*}$ is essentially constant.
\end{enumerate}
We have assumed that the simplicial set $W_{0,0}$ is discrete, so by condition \ref{cssinsecat3} each $W_{0,k}$ must be homotopy discrete.  However, by condition \eqref{cssinsecat1}, each of these simplicial sets must be the degree zero space of a complete Segal space.  Thus, these complete Segal spaces are essentially Segal categories, with the only variation that the degree zero spaces are homotopy discrete rather than actually discrete. We give a more detailed treatment of this phenomenon in \cite{discrete2}.
\end{remark}

\section{$\Theta_2$-models for $(\infty,2)$-categories}

The models in the previous section give us a number of ways to think about $(\infty,2)$-categories as weakly enriched categories.  However, there are reasons why we might want a different approach.  First, conditions such as the essential constancy of a complete Segal object suggest that somehow a bisimplicial diagram is bigger than what we need.  Furthermore, such models are not cartesian, a problem which led Rezk to look for an alternative.  Another approach which gives an answer to these difficulties is given via diagrams which are modeled by the category $\Theta_2$ rather than $\Delta \times \Delta$.  

Recall from the previous section that when modeling an $(\infty,2)$-category by a bisimplicial diagram, we have extra data, in that we do not want interesting ``vertical" morphisms.  So, the main idea is that we take diagrams of simplicial sets that more concisely model the data that we actually want.

Let us state the formal definition, then look at an example.

\begin{definition}
The category $\Theta_2$ has objects $[m]([k_1], \ldots, [k_m])$, where $[m]$ and each $[k_i]$ are objects of $\Delta$, and morphisms $[m]([k_1], \ldots, [k_m]) \rightarrow [p]([\ell_1], \ldots, [\ell_p])$ given by a function $\delta \colon [m] \rightarrow [p]$ in $\Delta$, and functions $[k_i] \rightarrow [\ell_j]$ defined whenever $\delta(i-1) < j \leq \delta(i)$.
\end{definition}

\begin{example}
The object $[4]\left([2], [3], [0], [1]\right)$ of $\Theta_2$ can be depicted as
\[ \xymatrix@1{0 \ar[r]^{[2]} & 1 \ar[r]^{[3]} & 2 \ar[r]^{[0]} & 3\ar[r]^{[1]} & 4}. \]
Since these labels, themselves objects of $\Delta$, can also be interpreted as strings of arrows, we get a diagram such as
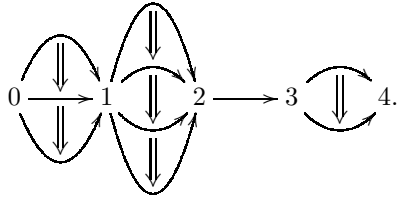
\[\xymatrix{ {0} \ar@/^2pc/[r]^{}="10"  \ar[r]^{}="11" \ar@/_2pc/[r]^{}="12" \ar@{=>}"10";"11" \ar@{=>}"11";"12"
& {1} \ar@/^3pc/[r]^{}="20" \ar@/^1pc/[r]^{}="21"
\ar@/_1pc/[r]^{}="22" \ar@/_3pc/[r]^{}="23" \ar@{=>}"20";"21"
\ar@{=>}"21";"22" \ar@{=>}"22";"23"
& {2} \ar[r]
& {3} \ar@/^1pc/[r]^{}="30" \ar@/_1pc/[r]^{}="31" \ar@{=>}"30";"31"
& {4.}
}\]
The elements of this diagram can be regarded as generating a strict 2-category by composing 1-cells and 2-cells whenever possible.  In other words, the objects of $\Theta_2$ can be seen as encoding all possible finite compositions that can take place in a 2-category, much as the objects of $\Delta$ can be thought of as listing all the finite compositions that can occur in an ordinary category.

Let us consider a morphism $[4]([2], [3], [0], [1]) \rightarrow [3]([1],[0], [2])$.  The first thing we need to define such a morphism is a map $\delta\colon [4] \rightarrow [3]$ in $\Delta$, for example the one depicted here:
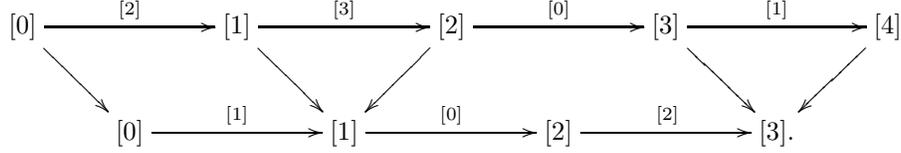
\[ \xymatrix{[0] \ar[rr]^{[2]} \ar[dr] && [1] \ar[rr]^{[3]} \ar[dr] && [2] \ar[rr]^{[0]} \ar[dl] && [3] \ar[rr]^{[1]} \ar[dr] && [4] \ar[dl] \\
& [0] \ar[rr]^{[1]} && [1] \ar[rr]^{[0]} && [2] \ar[rr]^{[2]} && [3]. &} \]
We have maps between the labeling objects when it is sensible to do so, as can be visualized in the above diagram.  For example, to complete the definition here we would need to specify a map $[2] \rightarrow [1]$, the unique map $[0] \rightarrow [0]$, and a map $[0] \rightarrow [2]$.  We think of the unique arrow $2 \rightarrow 3$ in the top diagram as being sent to the composite of the arrow $1 \rightarrow 2$ with the arrow $2 \rightarrow 3$ specified by the chosen map $[0] \rightarrow [2]$. 

More pictorally, one choice of such a morphism is given by:
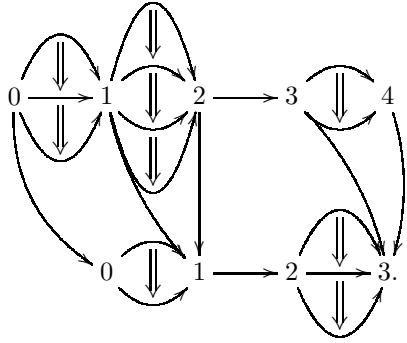
\[\xymatrix{ {0} \ar@/^2pc/[r]^{}="10" \ar[r]^{}="11" \ar@/_2pc/[r]^{}="12" \ar@{=>}"10";"11" \ar@{=>}"11";"12" \ar@/_1pc/@[red][ddr]
& {1} \ar@/^3pc/[r]^{}="20" \ar@/^1pc/[r]^{}="21"
\ar@/_1pc/[r]^{}="22" \ar@/_3pc/[r]^{}="23" \ar@{=>}"20";"21"
\ar@{=>}"21";"22" \ar@{=>}"22";"23"  \ar@/_.5pc/@[red][ddr]
& {2} \ar[r] \ar@[red][dd]
& {3} \ar@/^1pc/[r]^{}="30" \ar@/_1pc/[r]^{}="31" \ar@{=>}"30";"31" \ar@/^.5pc/@[red][ddr]
& {4} \ar@/^.5pc/@[red][dd] \\
\\
& {0} \ar@/^1pc/@[blue][r]^{}="30" \ar@/_1pc/@[blue][r]^{}="31" \ar@{=>}@[blue]"30";"31"
& {1}  \ar@[blue][r]
& {2} \ar@/^2pc/[r]^{}="10"  \ar[r]^{}="11" \ar@/_2pc/@[blue][r]^{}="12" \ar@{=>}"10";"11" \ar@{=>}"11";"12"
& {3.} } \]
Here the map choices are indicated by giving their image in blue; we have distinguished the map $\delta$ by the arrows in red.
\end{example}

To model $(\infty,2)$-categories, we want to consider functors $X \colon \Theta_2^{\op} \rightarrow \SSets$.  Our first question is how to describe the appropriate Segal conditions, of which we expect to have two, coming from the two ways that the category $\Delta$ appears in the category $\Theta_2$: the ``outside" indexing, and the ``internal" indexing, given by arrow labelings. We can conceptualize these kinds of conditions as follows.  If we apply such a functor $X$ to the object
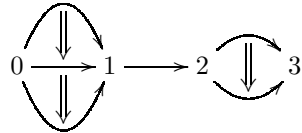
\[\xymatrix{ {0} \ar@/^2pc/[r]^{}="10"  \ar[r]^{}="11" \ar@/_2pc/[r]^{}="12" \ar@{=>}"10";"11" \ar@{=>}"11";"12"
& {1}  \ar[r]
& {2} \ar@/^1pc/[r]^{}="30" \ar@/_1pc/[r]^{}="31" \ar@{=>}"30";"31"
& {3}
}\]
of $\Theta_2^{\op}$, one Segal condition should give us that the resulting simplicial set should be weakly equivalent to the simplicial set
\[ X \left( \xymatrix{{0} \ar@/^2pc/[r]^{}="10"  \ar[r]^{}="11" \ar@/_2pc/[r]^{}="12" \ar@{=>}"10";"11" \ar@{=>}"11";"12"
& {1} }\right) \times_{X(0)} X \left(\xymatrix{{1}  \ar[r] & {2}} \right) \times_{X(1)} X\left( \xymatrix{{2} \ar@/^1pc/[r]^{}="30" \ar@/_1pc/[r]^{}="31" \ar@{=>}"30";"31"
& {3}} \right). \]
The other Segal condition should tell us that the simplicial set
\[ X \left( \xymatrix{{0} \ar@/^2pc/[r]^{}="10"  \ar[r]^{}="11" \ar@/_2pc/[r]^{}="12" \ar@{=>}"10";"11" \ar@{=>}"11";"12"
& {1} }\right) \]
is weakly equivalent to 
\[ X\left( \xymatrix{{0} \ar@/^1pc/[r]^{}="30" \ar@/_1pc/[r]^{}="31" \ar@{=>}"30";"31"
& {1}} \right) \times_{X(0 \rightarrow 1)}  X\left( \xymatrix{{0} \ar@/^1pc/[r]^{}="30" \ar@/_1pc/[r]^{}="31" \ar@{=>}"30";"31"
& {1}} \right).\]

Once again, we have choices to make as to whether we want to require completeness conditions or discreteness of certain component spaces.  Let us start by considering completeness for both, as Rezk did.   We still retain the analogue of completeness for a complete Segal space, namely that $X([0])$ is weakly equivalent to the subspace of homotopy equivalences in $X([1]([0])$, which we can visualize as
\[ X(0) \overset{\simeq}{\rightarrow} X \left(\xymatrix{{0}  \ar[r] & {1}} \right)_{\heq}.   \]
But we additionally have a higher-dimensional analogue, which says that $X([1]([0])$ is weakly equivalent to the subspace of homotopy equivalences in $X([1]([1])$, which can be visualized as
\[ X \left(\xymatrix{{0}  \ar[r] & {1}} \right) \overset{\simeq}{\rightarrow} X\left( \xymatrix{{0} \ar@/^1pc/[r]^{}="30" \ar@/_1pc/[r]^{}="31" \ar@{=>}"30";"31"
	& {1}} \right)_{\heq}    \]
	
Since $\Theta_2$ is still a Reedy category \cite{berger}, we can consider the category $\SSets^{\Theta_2^{\op}}$ with the Reedy model structure and localize so that the fibrant objects have the desired conditions.  

Given an object $[m]([k_1], \ldots, [k_m])$ of $\Theta_2$, with $m \geq 2$, let $\Theta[m]([k_1], \ldots, [k_m])$ be its associated representable functor $\Theta_2^{\op} \rightarrow \Sets$, thought of as a discrete functor $\Theta_2^{\op} \rightarrow \SSets$.  In analogy with the sub-simplicial space $G[n]^t \subseteq \Delta[n]^t$, define
\[ G[m]([k_1], \ldots, [k_m]) = \Theta[1]([k_1]) \amalg_{\Theta[0]} \cdots \amalg_{\Theta[0]} \Theta[1]([k_m]), \]
and observe that there is a natural inclusion map
\[ G[m]([k_1], \ldots, [k_m]) \rightarrow \Theta[m]([k_1], \ldots, [k_m]). \]
The ``horizontal", or ``outer", Segal condition can be made precise by asking that $X$ be local with respect to all such maps.

The second, ``internal" Segal condition is described similarly.  Given $k \geq 0$, define
\[ \Theta[1](G[k]) = \Theta[1]([1]) \amalg_{\Theta[1]([0])} \cdots \amalg_{\Theta[1]([0])} \Theta[1]([1]), \]
which naturally includes into $\Theta[1]([k])$.  Objects which are local with respect to these maps satisfy the second Segal condition.

As we have seen in other models, we can pause here and consider such functors with no further conditions.

\begin{definition}
A $\Theta_2$-\emph{Segal space} is a Reedy fibrant functor which satisfies both the Segal conditions described above.
\end{definition}

These objects have an associated model structure.
 
\begin{theorem} \cite{rezktheta}
There is a model structure on the category $\SSets^{\Theta_2^{\op}}$ in which the fibrant objects are precisely the $\Theta_2$-Segal spaces.
\end{theorem}

What about the completeness conditions?  For the first, we use an underlying simplicial space functor, just as we did for the multisimplicial model.  

\begin{definition}
Let $\tau_\Theta \colon \Delta \rightarrow \Theta_2$ be defined by $\tau_\Theta[m] = [m]([0], \ldots, [0])$.  The \emph{underlying simplicial space} of a functor $X \colon \Theta_2^{\op} \rightarrow \SSets$ is given by $\tau_\Theta^*(X)$.
\end{definition}

This functor $\tau_\Theta^*$ has a left adjoint, which we denote by $(\tau_\Theta)_\#$, To define our first completeness condition, we apply this functor to the map which we use to define complete Segal spaces.  Thus, we localize with respect to the map
\[ (\tau_\Theta)_\#(E^t) \rightarrow (\tau_\Theta)_\# \Delta[0]^t. \]

The second completeness condition is more subtle to define precisely.  The idea is to define an object $\Theta[E]$ which models homotopy 2-equivalences, and then localize with respect to the map
\[ \Theta[E] \rightarrow \Theta[1]([0]). \]
The object $\Theta[E]$ is defined via an intertwining functor which we do not describe here; we refer the reader to \cite[4.4]{rezktheta} for details.

\begin{definition}
A $\Theta_2$-\emph{space} is a Reedy fibrant functor $X \colon \Theta_2^{\op} \rightarrow \SSets$ such that these two Segal conditions and two completeness conditions hold. 
\end{definition}

\begin{theorem} \cite{rezktheta}
There is a cartesian model structure $\Theta_2\css$ on the category of all functors $\Theta_2^{\op} \rightarrow \SSets$ in which the fibrant objects are precisely the $\Theta_2$-spaces.
\end{theorem}

Once again, we want to have a notion of Dwyer-Kan equivalence between $\Theta_2$-spaces which is appropriately fully faithful and essentially surjective.  To do so, we need notions of mapping objects and a homotopy category.

\begin{definition}
Given a functor $X \colon \Theta_2^{\op} \rightarrow \SSets$ and any $(x, y) \in X[0]_0 \times X[0]_0$, we define the \emph{mapping object} $M_X^\Theta(x, y) \colon \Deltaop \rightarrow \SSets$, evaluated at the object $[k]$ of $\Delta$, as the pullback of the diagram
\[ \{(x, y)\} \rightarrow X[0]\times X[0] \leftarrow X[1]([k]). \]
\end{definition}

For essential surjectivity, we use the underlying simplicial space functor $\tau_\Theta^*$.

\begin{definition}
Let $X$ be a $\Theta_2$-Segal space.  Its \emph{homotopy category} $\Ho(X)$ has $X[0]_0$ as objects and 
\[ \Hom_{\Ho(X)}(x, y) = \Hom_{\Ho(\tau_\Theta^*X)}(x, y). \]
\end{definition}

Now we can make our definitions of fully faithful and essentially surjective.

\begin{definition}
Let $X$ and $Y$ be $\Theta_2$-Segal spaces. A morphism $f \colon X \rightarrow Y$ is \emph{fully faithful} if for every $x, y \in X[0]$ and every $k \geq 0$ the map
\[ M_X^\Theta(x, y)_k \rightarrow M_X^\Theta(fx, fy)_k \]
is a weak equivalence in $\SSets$.
\end{definition}

\begin{definition}
Let $X$ and $Y$ be $\Theta_2$-Segal spaces.  A morphism $X \rightarrow Y$ is \emph{essentially surjective} if $\Ho(f) \colon \Ho(X) \rightarrow \Ho(Y)$ is an essentially surjective functor of categories.
\end{definition}

We want to consider these notions for more general functors $\Theta_2^{\op} \rightarrow \SSets$, but we need to localize first. Let us denote by $LX$ the functorial localization of $X$ in the model structure for $\Theta_2$-Segal spaces.

\begin{definition}
Let $X$ and $Y$ be objects of $\SSets^{\Theta_2^{\op}}$.  A map $X \rightarrow Y$ is a \emph{Dwyer-Kan equivalence} if the associated map $LX \rightarrow LY$ is fully faithful and essentially surjective.  
\end{definition}

As in other models, we have the following result.

\begin{theorem} \label{dkwkequiv}
Let $X, Y \colon \Theta_2^{\op} \rightarrow \SSets$ be $\Theta_2$-Segal spaces.  A map $X \rightarrow Y$ is a Dwyer-Kan equivalence if and only if it is a weak equivalence in $\Theta_2\css$.
\end{theorem}

Now, we would like to compare this model structure to one of those previously developed to show that $\Theta_2$-spaces give good models for $(\infty,2)$-categories.  It is convenient to compare them to 2-fold complete Segal spaces, for which we need a way to relate the categories $\Theta_2$ and $\Delta \times \Delta$.

We define the functor $d \colon \Delta \times \Delta \rightarrow \Theta_2$ by
\[ ([m], [k]) \mapsto [m]([k], \ldots, [k]). \]
As the name suggests, we can think of $d$ as a kind of diagonal functor.  As an example, consider the object $([2], [1])$, which maps to $[2]([1], [1])$.  Visually, this assignment takes the diagram
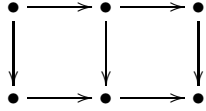
\[ \xymatrix{\bullet \ar[r] \ar[d] & \bullet \ar[r] \ar[d] & \bullet \ar[d] \\
\bullet \ar[r] & \bullet \ar[r] & \bullet} \]
to
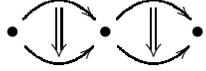
\[ \xymatrix{ { \bullet} \ar@/^1pc/[r]^{}="30" \ar@/_1pc/[r]^{}="31" \ar@{=>}"30";"31"
& {\bullet} \ar@/^1pc/[r]^{}="30" \ar@/_1pc/[r]^{}="31" \ar@{=>}"30";"31"
& {\bullet}.
}\]
Observe that this functor is modeling the kind of ``compression" of vertical morphisms that we wanted in the passage from a multisimplicial model to a $\Theta_2$-model.

The induced functor
\[ d^* \colon \SSets^{\Theta_2^{\op}} \rightarrow \SSets^{\Deltaop \times \Deltaop} \]
has both a left and a right adjoint, given by left and right Kan extension.  We are interested here in the right adjoint, which we denote by $d_*$.

\begin{theorem} \cite{inftyn2}
The adjunction $(d^*, d_*)$ induces a Quillen equivalence
\[ d^* \colon \Theta_2Sp \rightleftarrows \css(\css) \colon d_*. \]
\end{theorem}

This comparison shows us, in particular, that $\Theta_2$ does exactly what we want it to, in replacing the essential constancy condition for one simplicial direction of $\Delta \times \Delta$-diagrams by the compressed diagrams in $\Theta_2$.

One might observe, just as in the last section, that we have been giving preferential treatment to a model that incorporates completeness conditions, rather than asking that certain component spaces be discrete.  Let us remedy this situation now and consider such variants.

\begin{definition}
A $\Theta_2$-\emph{Segal precategory} is a functor $X \colon \Theta_2^{\op} \rightarrow \SSets$ for which the simplicial sets $X[0]$ and $X[1]([0])$ are discrete.  A $\Theta_2$-\emph{Segal category} additionally satisfies both Segal conditions.  
\end{definition}

\begin{theorem} \cite{discrete}
There is a model structure $\Theta_2\Secat$ on the category of $\Theta_2$-Segal precategories in which the fibrant objects are the Reedy fibrant $\Theta_2$-Segal categories.
\end{theorem}

But, as before, we have gone to the other extreme and required two discreteness conditions; we could instead mix and match the completeness and discreteness assumptions.  As with the multisimplicial models, only some of these combinations give distinct models.

\begin{definition}
A $\Theta_2$-\emph{Segal $[0]$-precategory} is a functor $X \colon \Theta_2^{\op} \rightarrow \SSets$ such that the simplicial set $X[0]$ is discrete.  It is a $\Theta_2$-\emph{Segal $[0]$-category} if additionally it satisfies both Segal conditions and the completeness condition that 
\[ X[1]([0]) \simeq X[1]([1])_{\heq}. \]
\end{definition}

\begin{theorem} \cite{discrete}
There is a model structure $\Theta_2\Se^{[0]}\Cat$ on the category of $\Theta_2^{\op}$-Segal $[0]$-precategories in which the fibrant objects are the Reedy fibrant $\Theta_2$-Segal $[0]$-categories.  
\end{theorem}

We can think of this model structure as giving an intermediate step between Segal 2-categories and 2-fold complete Segal spaces, as given by the following theorem.

\begin{theorem} \cite{discrete}
The inclusion functors and their right adjoints induce Quillen equivalences
\[ \Theta_2\Secat \rightleftarrows \Theta_2\Se^{[0]}\Cat \rightleftarrows
\Theta_2Sp. \]
\end{theorem}

If we try to make the opposite choice, making $X[1]([0])$ discrete and $X[0] \simeq X[1]([0])_{\heq}$, we actually force $X[0]$ to be discrete, since $X[1]([0])_{\heq}$ is a subspace of a discrete space and thus discrete, and $X[0]$ is a retract of it.  Thus, we recover the Segal 2-category model rather than something new.

The model structures that we have are all compatible with their respective bisimplicial models via the adjoint pair $(d^*, d_*)$.  The Quillen equivalences below not already mentioned are to appear in \cite{discrete2}.

\begin{theorem} 
The functors in the following commutative diagram are left adjoints of Quillen equivalences:
\[ \xymatrix{\Secat(\Secat) \ar[r] \ar[d] & \Secat(\css) \ar[r] \ar[d] & \css(\css) \ar[d] \\
\Theta_2\Secat \ar[r] & \Theta_2\Se^{[0]}\Cat \ar[r] & \Theta_2Sp.} \]
The horizontal functors are given by inclusion, whereas the vertical ones are $d_*$ or the appropriate restriction thereof.
\end{theorem}

\section{Generalizations of quasi-categories}

A natural question a reader might have at this point is whether there are generalizations of quasi-categories with the same kind of flavor.  This question was initially asked by Joyal, who tried to describe the analogues of horn-filling conditions for functors $\Theta_2^{\op} \rightarrow \Sets$ but was not successful.  With the advent of $\Theta_2$-spaces, however, Ara was able to exploit the method of proof for the Quillen equivalence between quasi-categories and complete Segal spaces to describe such structures and give a corresponding model structure which is Quillen equivalent to $\Theta_2Sp$ \cite{ara}. We sketch the main ideas here; the methods used are substantially different than than the ones used thus far in this paper, so we do not go into great detail.

Consider the category $\SSets^{\Theta_2^{\op}}$.  We can think of the objects of this category alternatively as functors $\Theta_2^{\op} \times \Deltaop \rightarrow \Sets$, or in turn as functors $\Deltaop \rightarrow \Sets^{\Theta_2^{\op}}$.  

We want to have a model structure on the category $\Sets^{\Theta_2^{\op}}$ such that the adjoint pair of functors between it and $\SSets^{\Theta_2^{\op}}$ given by inclusion and evaluation at simplicial degree zero gives a Quillen equivalence of model categories with $\Theta_2Sp$, in analogy with the Quillen equivalence between complete Segal spaces and quasi-categories.

The following definition is not particularly precise, but gives an idea of the flavor how we should think of a $\Theta_2$-set. Ara gives a more concrete definition in \cite{ara}.

\begin{definition}
A functor $K \colon \Theta_2^{\op} \rightarrow \Sets$ is a $\Theta_2$-\emph{set} if $K \cong \ev_0 X$ for some $\Theta_2$-space $X \colon \Theta_2^{\op} \rightarrow \SSets$.
\end{definition}

\begin{theorem} \cite[8.5]{ara}
There is a cartesian model structure $\Theta_2\Sets$ on the category of functors $\Theta_2^{\op}$ in which the fibrant objects are the $\Theta_2$-sets.
\end{theorem}

Ara's proof uses Cisinski's theory of $A$-localizers, as does the comparison with $\Theta_2$-spaces.  Here, we sketch an argument for the comparison which more closely models the one of Joyal and Tierney for quasi-categories and complete Segal spaces.

The next step is to give a model structure on the category $(\Sets^{\Theta_2^{\op}})^{\Deltaop}$ so that the fibrant objects are essentially constant in the simplicial direction, using the model structure for $\Theta_2$-spaces.  Denoting this model structure by $(\Theta_2\Sets)^{\Deltaop}_{ec}$, the following comparison is not unexpected.

\begin{theorem} \cite[8.4]{ara}
The inclusion and evaluation functors induce a Quillen equivalence
\[ \Theta_2\Sets \rightleftarrows (\Theta_2\Sets)^{\Deltaop}_{ec}. \]
\end{theorem}

Finally, one can check that the localization for this model structure agrees with the one for $\Theta_2$-spaces.

\begin{theorem}
The model structure $(\Theta_2\Sets)^{\Deltaop}_{ec}$ exactly coincides with the model structure $\Theta_2Sp$.
\end{theorem}

A natural question, given the comparisons of the previous section, is whether one can obtain a similar model using $\Delta \times \Delta$.  

\begin{question}
Is there an analogous model using bisimplicial sets?
\end{question}

\section{Generalizing to $(\infty,n)$-categories}

In this section we give a very brief introduction to generalizing the ideas of this paper to higher $(\infty,n)$-categories.  In some sense, the hard work is done in knowing how to generalize from $(\infty,1)$-categories to $(\infty,2)$-categories, and then we can proceed in an inductive way.  Of course, there are many variants and combinations, of which we can only give a hint here.

We describe two of the multisimplicial models inductively as follows.

\begin{definition}
A Reedy fibrant functor $(\Deltaop)^n \rightarrow \SSets$ is an $n$-\emph{fold complete Segal space} if it is a complete Segal object in $(n-1)$-fold complete Segal spaces.
\end{definition}

\begin{definition}
A functor $(\Deltaop)^n \rightarrow \SSets$ is a \emph{Segal $n$-category} if it is a Segal category object in $(n-1)$-Segal categories.
\end{definition}

Let us consider the generalization of $\Theta_2$ to $\Theta_n$.  We use the inductive approach of Berger \cite{berger}, which we have essentially used already in our definition of $\Theta_2$ above.  Given a small category $\mathcal C$, define the category $\Theta \mathcal C$ to have objects $[m](c_1, \ldots, c_m)$ where $[m]$ is an object of $\Delta$ and $c_1, \ldots, c_n$ are objects of $\mathcal C$.  A morphism between two such objects is defined similarly as for the ones in $\Theta_2$.

Let $\Theta_0$ be the category with one object and only an identity morphism, and inductively define $\Theta_n = \Theta \Theta_{n-1}$.  Observe that $\Theta_1 = \Delta$, and $\Theta_2$ is exactly the category as we described it previously.

Consider a functor $\Theta_n^{\op} \rightarrow \SSets$.  As we did for $\Theta_2$-spaces, we can define $n$ different Segal conditions: one ``outermost" or ``horizontal" condition, given by the inclusions
\[ G[m](c_1, \ldots, c_m) \rightarrow \Theta[m](c_1, \ldots, c_m) \]
for any object $[m](c_1, \ldots c_m)$ of $\Theta_n$, so that each $c_i$ is an object of $\Theta_{n-1}$.  Then the other $n-1$ Segal conditions can be imported inductively from those for $\Theta_{n-1}$-spaces.

Completeness is similar: the completeness condition which says that the space of objects is weakly equivalent to the space of homotopy equivalences is given by the usual completeness condition on the underlying simplicial space.  The higher completeness conditions are given by incorporating those for $\Theta_{n-1}$-spaces.

\begin{definition}
A $\Theta_n$-\emph{space} is a Reedy fibrant functor $\Thetanop \rightarrow \SSets$ which satisfies these $n$ Segal conditions and $n$ completeness conditions.
\end{definition}

\begin{theorem} \cite{rezktheta}
There is a cartesian model structure $\Theta_n\css$ on the category $\SSets^{\Thetanop}$ in which the fibrant objects are the $\Theta_n$-spaces.
\end{theorem}

Because this model structure is cartesian, we can hope that there is a corresponding model structure for categories enriched in it.  The proof that we have a model structure on small categories enriched in $\css$ extends to this case.  Then our chain of Quillen equivalences between $\css$-$\Cat$ and $\Theta_2\css$ generalizes to the following result.

\begin{theorem} \cite{inftyn1}, \cite{inftyn2}
There is a chain of Quillen equivalences between $\Theta_{n-1}Sp$-$\Cat$ and $\Theta_n\css$.
\end{theorem}

In particular, we have appropriate notions of complete Segal objects in $\Theta_n$-spaces and Segal category objects in $\Theta_n$-spaces.  The definitions are quite similar to the ones we have given for $n=2$, so we look instead at some interesting things that happen for $n \geq 3$.

Recall the diagonal functor $d \colon \Delta \times \Delta \rightarrow \Theta_2$ which we used to compare 2-fold complete Segal spaces and $\Theta_2$-spaces.  In the case of higher $n$, this functor actually gives us a means of going from $\Delta^n \rightarrow \Theta_n$ incrementally.  

\begin{theorem} \cite{inftyn2}
Define the functor $d \colon \Delta \times \Theta_{n-1} \rightarrow \Theta_n$ by $([m], c) \mapsto [m](c, \ldots, c)$.  Then the adjoint pair
\[ d^* \colon \SSets^{\Thetanop} \rightleftarrows \SSets^{\Deltaop \times \Thetanop} \colon d_* \]
induces a Quillen equivalence between the model structure for $\Theta_n$-spaces and the model structure for complete Segal objects in $\Theta_{n-1}$-spaces.
\end{theorem}

We can continue this process, successively picking copies of $\Delta$ off of $\Theta_n$ via the functor $d$:
\[ \underbrace{\Delta \times \cdots \times \Delta}_n \rightarrow \underbrace{\Delta \times \cdots \Delta}_{n-2} \times \Theta_2 \rightarrow \cdots \rightarrow \Delta \times \Theta_{n-1} \rightarrow \Theta_n. \]

\begin{theorem} \cite{inftyn2}
The above chain of functors induces Quillen equivalences between a model structure for $n$-fold complete Segal spaces and $\Theta_n\css$.
\end{theorem}

Another proof of this comparison, not using model categories, is given by Haugseng \cite{haug}.

The various models in between can be thought of as hybrids between the $\Theta_n$-space model and the multisimplicial model.  In all these cases, one can also replace completeness conditions with discreteness conditions, but as we saw for $(\infty,2)$-categories, we only get distinct models when we choose discreteness models from the bottom up.  So, for example, we could ask that the spaces of objects and of $i$-morphisms, for $i \leq k$, be discrete, and for the spaces of $i$-morphisms, for $k<i<n$, to be weakly equivalent to the space of $i+1$-homotopy equivalences.  We discuss these models, and the comparisons between them, in \cite{discrete} and \cite{discrete2}.

Ara's proof of the comparison between $\Theta_2$-sets and $\Theta_2$-spaces extends to an analogous comparison between $\Theta_n$-sets and $\Theta_n$-spaces.  Again, a natural question is whether there is an analogous way to develop a multisimplicial model, or models corresponding to the various interpolations between $\Delta \times \cdots \times \Delta$ and $\Theta_n$.


\begin{thebibliography}{99}

\bibitem{ara}
Dimitri Ara, Higher quasi-categories vs higher Rezk spaces,  \emph{J.\ K-Theory} 14 (2014), no.\ 3, 701--749.

\bibitem{arahh}
Dimitri Ara, On the homotopy theory of Grothendieck $\infty$-groupoids, \emph{J.\ Pure Appl.\ Algebra} 217 (2013) 1237--1278.

\bibitem{afr}
David Ayala, John Francis, and Nick Rozenblyum, Factorization homology I: Higher categories, \emph{Adv.\ Math.} 333 (2018), 1042--1177.

\bibitem{barwick}
Clark Barwick, On left and right model categories and left and right Bousfield localizations, \emph{Homology, Homotopy Appl.\ } 12 (2010), no.\ 2, 245--320.

\bibitem{bk}
C.\ Barwick and D.M.\ Kan, Relative categories: Another model for the homotopy theory of homotopy theories, \emph{Indag.\ Math.\ (N.S.)} 23 (2012), no.\ 1--2, 42--68.

\bibitem{bknrel}
C.\ Barwick and D.M.\ Kan, n-relative categories: a model for the homotopy theory of n-fold homotopy theories. \emph{Homology Homotopy Appl.} 15 (2013), no.\ 2, 281--300.

\bibitem{bsp}
Clark Barwick and Christopher Schommer-Pries, On the unicity of the homotopy theory of higher categories, preprint available at math.AT/1112.0040.

\bibitem{berger}
Clemens Berger, Iterated wreath product of the simplex category and iterated loop spaces, \emph{Adv.\ Math.\ } 213 (2007) 230--270.

\bibitem{simpcat}
Julia E.\ Bergner, A model category structure on the category of simplicial categories, \emph{Trans.\ Amer.\ Math.\ Soc.\ } 359 (2007), 2043--2058.

\bibitem{discrete}
Julia E.\ Bergner, Models for $(\infty,n)$-categories with discreteness conditions, I, in preparation.

\bibitem{discrete2}
Julia E.\ Bergner, Models for $(\infty,n)$-categories with discreteness conditions, II, in preparation.

\bibitem{thesis}
Julia E.\ Bergner, Three models for the homotopy theory of homotopy theories, \emph{Topology} 46 (2007), 397--436.

\bibitem{boors3}
Julia E.\ Bergner, Ang\'elica M.\ Osorno, Viktoriya Ozornova, Martina Rovelli, and Claudia I.\ Scheimbauer, 2-Segal objects and the Waldhausen construction, in preparation.

\bibitem{inftyn1}
Julia E.\ Bergner and Charles Rezk, Comparison of models for $(\infty, n)$-categories, I, \emph{Geom.\ Topol.} 17 (2013) 2163--2202.

\bibitem{inftyn2}
Julia E.\ Bergner and Charles Rezk, Comparison of models for $(\infty, n)$-categories, II, preprint available at math.AT/1406.4182.

\bibitem{bpsegal}
David Blanc and Simona Paoli, Segal-type algebraic models of $n$-types, \emph{Algebr.\ Geom.\ Topol.} 14 (2014), no.\ 6, 3419--3491. 

\bibitem{cisinski}
Denis-Charles Cisinski,  Batanin higher groupoids and homotopy types, \emph{Categories in algebra, geometry and mathematical physics, Contemp.\ Math.}, 431,  171--186.

\bibitem{cp}
J.M.\ Cordier and T.\ Porter, Vogt's theorem on categories of homotopy coherent diagrams, \emph{Math.\ Proc.\ Camb.\ Phil.\ Soc.\ } (1986), 100, 65--90.

\bibitem{dugspims}
Daniel Dugger and David I.\ Spivak, Mapping spaces in quasicategories, \emph{Algebr.\ Geom.\ Topol.\ } 11 (2011) 263--325.

\bibitem{dugspir}
Daniel Dugger and David I.\ Spivak, Rigidification of quasicategories, \emph{Algebr.\ Geom.\ Topol.\ } 11 (2011) 225--261.

\bibitem{dks}
W.G.\ Dwyer, D.M.\ Kan, and J.H.\ Smith,  Homotopy commutative diagrams and their realizations.  \emph{J.\ Pure Appl.\ Algebra} 57(1989), 5--24.

\bibitem{ds}
W.G.\ Dwyer and J.\ Spalinski, Homotopy theories and model categories, in \emph{Handbook of Algebraic Topology}, Elsevier, 1995.

\bibitem{gs}
Paul Goerss and Kristen Schemmerhorn, Model categories and simplicial methods. \emph{Interactions between homotopy theory and algebra, Contemp. Math.}, 436 (2007)  3–49.

\bibitem{haug}
Rune Haugseng, On the equivalence between $\Theta_n$-spaces and iterated Segal spaces. \emph{Proc.\ Amer.\ Math.\ Soc.} 146 (2018), no.\ 4, 1401--1415.

\bibitem{hirsch}
Philip S.\ Hirschhorn, \emph{Model Categories and Their Localizations, Mathematical Surveys and Monographs 99}, AMS, 2003.

\bibitem{hs}
A.\ Hirschowitz and C.\ Simpson,  Descente pour les $n$-champs, preprint available at math.AG/9807049.

\bibitem{hovey}
Mark Hovey, \emph{Model Categories, Mathematical Surveys and Monographs, 63}. American Mathematical Society 1999.

\bibitem{ilias}
Amrani Ilias, Model structure on the category of small topological categories. \emph{J.\ Homotopy Relat.\ Struct.} 10 (2015), no.\ 1, 63--70.

\bibitem{jfs}
Theo Johnson-Freyd and Claudia Scheimbauer, (Op)lax natural transformations, twisted quantum field theories, and "even higher'' Morita categories. \emph{Adv.\ Math.} 307 (2017), 147--223.

\bibitem{joyal}
A.\ Joyal, Simplicial categories vs quasi-categories, in preparation.

\bibitem{joyalbook}
A.\ Joyal, The theory of quasi-categories I, in preparation.

\bibitem{jt}
Andr\'{e} Joyal and Myles Tierney, Quasi-categories vs Segal spaces, \emph{Contemp.\ Math.\ } 431 (2007) 277--326.

\bibitem{leinstersurvey}
Tom Leinster, A survey of definitions of $n$-category,  \emph{Theory Appl.\ Categ.} 10 (2002), 1--70.

\bibitem{lurie}
Jacob Lurie, \emph{Higher topos theory. Annals of Mathematics Studies}, 170. Princeton University Press, Princeton, NJ, 2009.

\bibitem{luriegc}
Jacob Lurie, $(\infty, 2)$-categories and Goodwillie calculus, preprint available at math.CT/09050462.

\bibitem{lurietft}
Jacob Lurie, On the classification of topological field theories. \emph{Current developments in mathematics}, 2008, 129--280, Int. Press, Somerville, MA, 2009.

\bibitem{or}
Viktoriya Ozornova and Martina Rovelli, Model structures for $(\infty,n)$-categories on (pre)stratified simplicial sets and prestratified simplicial spaces, preprint available at math.AT/1809.10621.

\bibitem{pel}
Regis Pelissier, Cat\'egories enrichies faibles, preprint available at math.AT/0308246.

\bibitem{quillen}
Daniel Quillen, \emph{Homotopical Algebra, Lecture Notes in Math 43}, Springer-Verlag, 1967.

\bibitem{reedy}
C.L.\ Reedy, Homotopy theory of model categories, available at www-math.mit.edu/\verb1~1psh.

\bibitem{rezktheta}
Charles Rezk, A cartesian presentation of weak $n$-categories, \emph{Geom.\ Topol.\ } 14 (2010) 521--571.

\bibitem{rezk}
Charles Rezk, A model for the homotopy theory of homotopy theory, \emph{Trans.\ Amer.\ Math.\ Soc.\ }, 353(3), 973--1007.

\bibitem{rvbook}
Emily Riehl and Dominic Verity, Elements of $\infty$-Category Theory, preprint available at http://www.math.jhu.edu/\verb1~1eriehl/elements.pdf.

\bibitem{rv}
Emily Riehl and Dominic Verity, Fibrations and Yoneda's lemma in an $\infty$-cosmos. \emph{J.\ Pure Appl.\ Algebra} 221 (2017), no.\ 3, 499--564.

\bibitem{simpson}
Carlos Simpson, \emph{Homotopy Theory of Higher Categories, New Mathematical Monographs} 19. Cambridge University Press, Cambridge, 2012.

\bibitem{toen}
Bertrand To\"{e}n, Vers une axiomatisation de la th\'{e}orie des cat\'{e}gories sup\'{e}rieures, \emph{K-Theory} 34 (2005), no.\ 3, 233--263.

\bibitem{verity}
D.R.B.\ Verity, Weak complicial sets I: Basic homotopy theory.  \emph{Adv.\ Math.\ } 219 (2008), no.\ 4, 1081--1149.
\end{thebibliography}
\end{document}